\documentclass[11pt]{amsart}
\usepackage{cmap}
\usepackage[cp1251]{inputenc}
\usepackage[english]{babel}
\usepackage{amsmath}
\usepackage{amssymb}
\usepackage{amsfonts}
\usepackage{tikz}
\usepackage[shortlabels]{enumitem}

\newtheorem{lemma}{Lemma}

\newtheorem*{theorem*}{Theorem}

\newtheorem{corollary}{Corollary}
\newtheorem{proposition}{Proposition}

\newcommand{\gen}[1]{\langle #1 \rangle}
\newcommand{\disjoint}{\dot{\cup}}

\title[About chromatic uniqueness]{About chromatic uniqueness of some complete tripartite
    graphs}
\author{{P. A. Gein}}
\address{P. A. Gein
\newline\hphantom{iii} Institute of Natural Sciences and Mathematics,
\newline\hphantom{iii} Ural Federal University,
\newline\hphantom{iii} Lenina, 51,
\newline\hphantom{iii} 62083, Ekaterinburg, Russia}%
\email{pavel.gein@gmail.com}%

\thanks{This paper is an English version of the paper \cite{Gein2}}

\begin{document}

\begin{abstract}
Let $P(G, x)$ be the chromatic polynomial of a graph $G$. A graph $G$ is called \textit{chromatically unique}
if for any graph $H,\, P(G, x) = P(H, x)$ implies that $G$ and $H$ are isomorphic.
In this paper we show that full tripartite graph
$K(n_1, n_2, n_3)$ is chromatically unique
if $n_1 \geq n_2 \geq n_2 \geq n_3 \geq 2, n_1 - n_3 \leq 5$
and  $n_1 + n_2 + n_3 \not \equiv  2 \mod{3}$.
\end{abstract}

\maketitle

\section{Introduction}
In this paper all graphs are considered to be simple, that is they do not contain loops and multiple edges. Terminology is used with accordance to \cite{ABR}.

A \textit{(proper) coloring of a graph $G$ in $x$ colors} is a map $\phi$
from the set of all vertices of the graph $G$ to the set of numbers $\{1, 2, \ldots x\}$,
such as $\phi(u) \neq \phi(v)$ holds for any two adjacent vertices $u$ and $v$. A graph is called
\textit{$x$-colorable} if there exists its
coloring in $x$ colors. Denote the number of all colorings of the graph $G$ in $x$ colors as $P(G, x)$.
It is well known (see, for example, \cite{ABR}), that the function $P(G, x)$ is a polynomial,
which is called \textit{the chromatic polynomial of the graph $G$}.
Two graphs are called \textit{chromatically equivalent} if its chromatic polynomial are coincide.
A graph $G$ is called \textit{chromatically unique} if for any graph $H,\, P(G, x) = P(H, x)$ implies that
$G$ and $H$ are isomorphic.

The following question is especially interesting:
is any complete $t$-partite graph $K(n_1, n_2, \ldots, n_t)$ chromatically
unique whenever $t\geq 3$ and $n_1\geq n_2\geq \ldots \geq n_t \geq 2$?

List some known results, any additional details one can find in the book \cite{DongKohTeo} and
in the monograph \cite{Zhao}.
\begin{enumerate}
    \item A graph $K(n_1, n_2)$ is chromatically unique if $n_1\geq n_2\geq 2$, see \cite{KohTeo}.
    \item A graph $K(n_1, n_2, n_3, \ldots n_t)$ is chromatically unique if $t\geq 3$ and $n_1\geq n_2\geq \ldots n_t\geq2$ and $n_1 - n_t\leq 4$,
    see \cite{KorolevaR0, KorolevaR1, KorolevaR2, BarSenDC}.
    \item A graph $K(n_1, n_1, n_3)$ is chromatically unique if $n_1 - 1\geq n_3 \geq2$, see \cite{LiuZhaoYe2004}.
    \item A graph $K(n_1, n_1 - 1, n_3)$ is chromatically unique if $n_1 - 1 \geq n_3 \geq2$, see \cite{Gein}.
\end{enumerate}

The main result of this paper is the following
\begin{theorem*}
    A graph $K(n_1, n_2, n_3)$ is chromatically unique
    if $n_1\geq n_2\geq n_3\geq 2, n_1 - n_3\leq 5$ and $n_1 + n_2 + n_3 \not \equiv  2 \mod{3}$.
\end{theorem*}

Chromatically uniqueness of a graph $K(n_1, n_2, n_3)$, where $n_1 - n_3\leq 4$ was proved in \cite{KorolevaR0, KorolevaR1, KorolevaR2}.
The main aim of this paper is proving the theorem in the case when $n_1 - n_3 = 5$.

\section{Preliminaries}
A \textit{partition of a number $n$} is a sequence of nonnegative integers
$u = (u_1, u_2, \ldots)$ such that $u_1 \geq u_2 \ldots$, and
$u$ contains only finite non-zero elements,
and $n = \sum_{i=1}^{\infty} u_i$. The \textit{length} of the partition $u$ is the number $l$,
such that $u_l > 0$ and $u_{l + 1} = u_{l + 2} = \ldots = 0$.
When we write a partition, we will often omit its zero elements.

Let $u = (u_1, u_2, \ldots)$ and $v=(v_1, v_2, \ldots)$ be two partitions of a number $n$.
Then $v \trianglelefteq u$ if
\begin{align*}
    v_1 &\leq u_1, \\
    v_1 + v_2 &\leq u_1 + u_2, \\
              &\ldots \\
    v_1 + v_2 + \ldots + v_{t-1} &\leq u_1 + u_2 + \ldots + u_{t-1},
\end{align*}
where $t$ is the greatest of lengths $u$ and $v$.
The relation $\trianglelefteq$ is called \textit{dominance order}.
As it was shown in \cite{Brylawski}, all partitions of the number $n$
form a lattice with respect to $\trianglelefteq$.

As it was proved in \cite{lattice}, all partitions of the number $n$ with fixed length form a lattice
with respect to $\trianglelefteq$. Also Baransky and Sen'chonok in \cite{lattice}
introduce a notion of \textit{an elementary transformation}.
A partition $v = (v_1, v_2, \ldots v_t)$ is a result of an elementary transformation of
a partition $u = (u_1, u_2, \ldots, U_t)$, if there are such indicies $i$ and $j$ that
1)~$1\leq i < j \leq t$,
2)~$u_i - 1\geq u_{i+1}$ and $u_{j-1}\geq u_{j} + 1$,
3)~$u_i - u_j = \delta \geq 2$,
4)~$v_i = u_i - 1, v_j = u_j + 1, u_k = v_k$ for all $k=1, 2,\ldots, t, k\neq i, j$.
It was proved in \cite{lattice}, that $v\trianglelefteq u$ holds if and only if the partition $v$ can be obtained
from the partition $u$ with finite number elementary transformations.

Every complete $t$-partite graph with $n$ vertices can be identified with partition of length $t$ of
the number $n$.
Let $u = (u_1, u_2, \ldots, u_t)$ be a partition of length $n$ of the number $n$.
We will write $K(u)$ instead of $K(u_1, u_2, \ldots, u_t)$ and
denote parts of graph $K(u)$ as $V_i$ where $|V_i| = u_i$ for all $i=1, 2, \ldots t$.

Let $u$ be a partition of a number $n$ of length $t$.
We present the following schema for proving chromatic uniqueness of the graph $K(u)$.
By contradiction, we assume that the graph $K(u)$ is not chromatically unique.
It means, that there exists a graph $H$, which is nonisomorphic to the graph $K(u)$,
and graphs $H$ and $K(u)$ are chromatically equivalent.
It is clear, that the chromatic number of the graph $H$ is equal to $t$, so the graph $H$ can be obtained
from some complete $t$-partite graph by deleting some set of edges $E$.
It was shown in \cite{ZhaoRoots}, that different complete multipartite graphs
are not chromatically equivalent, so $E$ must be non empty.

Assume that some number is assigned to every graph. This number is called \textit{a chromatic invariant}
if it is the same for all chromatic equivalent graphs.
If $\alpha(G)$ is a chromatic invariant and $G_1$, $G_2$ are two arbitrary graphs, than denote
$\Delta \alpha(G_2, G_1) = \alpha(G_2) - \alpha(G_1)$. It is well known (see, for example, \cite{ABR}),
that the number of vertices, the number of edges,
the number of connected components and the number of triangles are chromatic invariants.

According to the Zykov's theorem (see, for example, \cite{ABR}),
the chromatic polynomial can be written as $P(G, x) = \sum\limits_{i=\chi}^{n} pt(G, i)x^{(i)}$,
where $pt(G, i)$ is a number of way to partition the vertex set of the graph $G$
into $t$ independent set, and $x^{(i)}$ is a factorial power of number $x$,
that is $x^{(i)} = x(x - 1)\cdot\ldots (x - i + 1)$.
It follows from Zykov's theorem, that numbers $pt(G, i), i=\chi, \ldots, n$ are chromatic
invariants. We are mostly interested in $pt(G, \chi + 1)$, which we will write as $pt(G)$.

It is clear, that every complete $t$-partite graph is $t$-colorable, but is not $(t-1)$-colorable;
in other words, the chromatic number of complete $t$-partite graph is equal to $t$.
Compute $pt(K(u))$ for complete multipartite graph $K(u_1, u_2, \ldots, u_t)$.
It is easy to show, that any partition of the vertex set of the graph $K(u)$ into $t + 1$ parts
can be obtained by splitting exactly one part into two nonempty subsets;
so $pt(K(u)) = \sum\limits_{i=1}^n 2^{u_i - 1} - t$.

It was investigated in \cite{BarSenDC}, how invariant $pt$ changes from graph $K(v)$ to graph $H$.
Introduce all necessary definitions and auxiliary statements.

A complete multipartite subgraph $G_1$ of the graph $K(v)$ is called \textit{$E$-subgraph},
if every part of the graph $G_1$ is contained in some part of the graph $K(v)$,
and the edge set of graph $G_1$ is contained in the set $E$.
An arbitrary disjoint set of $E$-subgraphs is called \textit{a garland}.
We will say that the garland $G'$ destroys a part $V_i$,
if every vertex of $V_i$ is contained in some $E$-subgraph of the garland $G'$.
A garland of cardinality $p$, which destroys exactly $p-1$ parts,
is called \textit{interesting}.
The set of all edges of all $E$-subgraph of the garland is called \textit{edge aggregate}.
A garland is called $k$-edge if its edge aggregate contains exactly $k$ edges.
Following properties was proved in \cite{BarSenDC}.
\begin{enumerate}[1)]
    \item If the chromatic number of the graph $H$ is equal to $t$,
    then every garland of cardinality $p$ destroys at most $p-1$ parts.
    \item Each garland is uniquely defined by its edge aggregate.
    \item A number $\Delta pt(H, K(v))$ is equal to the number of all interesting garlands.
\end{enumerate}

The next lemma follows from this properties.
\begin{lemma}[Corollary 2, \cite{BarSenDC}]\label{lemma:PT_Main}
    If a graph $H$ is obtained from graph $K(v)$ by deleting some set of edges $E$,
    and graphs $K(u)$ and $H$ are chromatically equivalent,
    then $|E| \leq \Delta pt(H, K(v)) \leq 2^{|E| - 1}$.
\end{lemma}

Let $G' = \{G'_1, G'_2,\ldots, G'_p\}$ be a garland.
We will say, that garland $G'$ has type $H_1\disjoint H_2\disjoint\ldots\disjoint H_p$,
where $\{H_1, H_2, \ldots H_p\}$ is a set of graphs, if $G'_i \simeq H_I$ for all $i=1, 2, \ldots p$.
Denote a number of interesting garlands, which edge aggregates contain exactly $k$ edges, as $\mu_k$.

Let $e$ be an arbitrary edge from $E$.
Denote the number of triangles of the graph $K(v)$, which contain edge $e$, as $\xi_1(e)$.
Let $\xi_i = \sum_{e\in E} \xi_1(e)$.

Consider a triangle in the graph $G$, which contains exactly two edges from $E$. Denote them as $e_1$ and $e_2$.
Subgraph, generated by $\{e_1, e_2\}$ is called a $\Xi_2$-subgraph.
Denote the number of such subgraphs as $\xi_2$.
Denote the number of triangles in $\gen{E}$ as $\xi_3$.

Denote the number of triangles in the graph $G$ as $I_3(G)$.
In \cite{KorolevaR0} the equation $\Delta I_3(K(v), H) = \xi_1 - \xi_2 - 2\xi_3$ was established.
Notice, when an edge is deleted, a new triangle can not be produced,
so $\Delta I_3(K(v), H)$ is equal to the number
of triangles in $K(v)$, which are destroyed by deleting edge set $E$ from $K(v)$.

The following lemma shows a connection between the number of interesting two-edge garland,
$\Xi_2$-subgraphs and the number of triangles in the graph $\gen{E}$.

\begin{lemma}\label{EdgePairsLemma}
    Let each part in the graph $K(v)$ contains at least three vertices and the edge set $E$ was deleted.
    Let $(d_1, d_2, \ldots, d_k)$ be a sequence of degree of vertexes of the graph $\gen{E}$. Then
    $$
    \mu_2 + \xi_2 + 3\xi_3 = \sum\limits_{i=1}^k \binom{d_i}{2} \leqslant \binom{|E|}{2}.
    $$
\end{lemma}

\begin{proof}
    Because each part contains at least three vertices, an interesting two-edge garland has type $K(2, 1)$,
    because a pair of nonadjacent edges can not be edge aggregate of any interesting garland,
    since such garland should destroy some part, which is impossible.

    Consider an arbitrary pair of adjacent edge. It either generates an interesting garland, forms $\Xi_2$-subgraph, or lays in some triangle. It is clear,
    that every triangle will be counted three times and subgraphs of two other types will be counted exactly once. The number of pair of adjacent edges
    is equal to $\sum\limits_{i=1}^k \binom{d_i}{2}$.
\end{proof}

Investigate the case, when inequality from lemma~\ref{EdgePairsLemma} become an equality. It is possible if and only if, when any two edges in $\gen{E}$
are adjacent.

Let $G_1 = (VG_1, EG_1), G_2 = (VG_2, EG_2)$ be two graphs. Define a graph $G_1 + G_2$ using following relations:
\begin{gather*}
    V(G_1+G_2) = VG_1 \disjoint VG_2 \\
    E(G_1+G_2) = EG_1 \disjoint EG_2 \disjoint \left\{\{x,y\} | x\in VG_1, y\in VG_2 \right\}.
\end{gather*}

Denote a graph with $n$ vertices without any edges as $O_n$.

\begin{lemma}
    Let $G$ be a graph without isolated vertices which has $m$ edges and any two edges in $G$ are adjacent. Then $G$ is isomorphic either to triangle
    or to the graph $O_m + O_1$.
\end{lemma}

\begin{proof}
    It is clear, that there are no cycles of length greater than 3 in the graph $G$ (in other case, there is a pair of nonadjacent edges).
    If there is a triangle in the graph $G$, then there are no other edges, because in other case such edge should go through two vertices of the triangle,
    but in this case the graph $G$ contains multiple edges.

    The last case to be considered, when there are no cycles in the graph $G$, so $G$ is a tree. Let $x$ be a leaf and it is adjacent with a vertex $y$.
    Then all another edges (if they exist) should go through the vertex $y$, therefore, the graph
    $G$ isomorphic to the graph $O_m + O_1$.
\end{proof}

\textbf{Remark.}
    It is clear, that graphs $O_m + O_1$ and $K(m, 1)$ are isomorphic. We well write, that a subgraph of the graph $\gen{E}$ is
    \textit{coordinated subgraph of type $K(m, 1)$}, if it is isomorphic to $K(m, 1)$ and all its $m$ vertices degree one lay in the same part of the graph
    $K(v)$.

Subset $E_1$ of the set $E$ is called \textit{uncontinuable}, if there is no garland in $\gen{E}$, which contains all edges from $E_1$.
In other case, subset is called \textit{continuable}. Remark, that an empty set is continuable.

Let $E_1$ be a subset of $E$. Subset $E_2$ of the set $E_1$ is called \textit{continuable outside of $E_1$},
if there exists garland $G'$ with edge aggregate $E'$, such that $E_2 = E'\cap E_1$.

\begin{lemma}\label{UncontLemma}
    Let $E_1\subset E$ and $E_1$ contains at most $N$ continuable outside of $E_1$ subsets.
    Then the number of garlands is not greater than $N\cdot 2^{|E| - |E_1|} - 1$.
\end{lemma}

\begin{proof}
    Let $X$ be a set of all garlands. Consider an arbitrary garland with edge aggregate $\hat{E}$. Notice, that a set of edges
    $E' = \hat{E}\cap E_1$ is continuable outside of $E_1$ subset.

    Consider an arbitrary subset $E'\subseteq E_1$, which is continuable outside of $E_1$. Let $X(E', E_1)$ be a set of all garlands,
    such that an intersection of their edge aggregate and set $E_1$ is equal to $E'$.
    Then $|X(E', E_1)|\leq 2^{|E| - |E_1|}$. Given the fact that $X=\disjoint_{E'} X(E', E_1)$,
    one can deduce $|X| = \sum\limits_{E'} |X(E', E_1)| = X(\varnothing, E_1) + \sum\limits_{E'\neq\varnothing} |X(E', E_1)|
    \leqslant 2^{|E| - |E_1|} - 1 + \sum\limits_{E'\neq\varnothing} 2^{|E| - |E_1|}  = N\cdot 2^{|E|-|E_1|} - 1$.
\end{proof}

The next three lemmas follow from lemma~\ref{UncontLemma}.

\begin{lemma}\label{TriagLemma}
    If there is a triangle in $\gen{E}$, then the number of garlands does not exceed $5\cdot 2^{|E| - 3} - 1$.
\end{lemma}

\begin{proof}
    Notice, that a triangle has at most 5 continuable outside itself subsets: empty, 3 one-edge subsets and the triangle.
\end{proof}

\begin{lemma}\label{Xi2Lemma}
    \begin{enumerate}
        \item If there is a subgraph of type $\Xi_2$ in $\gen{E}$,
        then the number of garland does not exceed $3\cdot 2^{|E| - 2} - 1$.
        \item If there are two distinct subgraphs of type $\Xi_2$ with edge sets $E_1$ and $E_2$,
        then the number of garlands does not exceed $2^{|E| - 1} + 2^{|E| - |E_1 \cup E_2|} - 1$.
    \end{enumerate}
\end{lemma}

\begin{proof}
    \begin{enumerate}
        \item Notice, that the $\Xi_2$-subgraph edge set  has at most 3 continuable outside itself subsets:
            empty and two one-edge subsets.
        \item Notice, that edge aggregate  of any garland can not contain nor set $E_1$, nor set
        $E_2$
        (since  a garland is a disjoint union complete multipartite graphs,
         then if it contains edges $xy$ and $xz$ of $\Xi_2$-subgraph,
         it also should contain an edge $yz$, which does not lay in $E$, see fig.~\ref{fig:Xi2}).
        Then by inclusion-declusion principle, the number of garlands does not exceed
        $2^{|E|} - 2^{|E| - |E_1|} - 2^{|E| - |E_2|} + 2^{|E| - |E_1\cup E_2|} - 1 = 2^{|E|} - 1 - 2\cdot 2^{|E| - 2} + 2^{|E| - |E_1 \cup E_2|} =
        2^{|E| - 1} + 2^{|E| - |E_1 \cup E_2|} - 1$.
        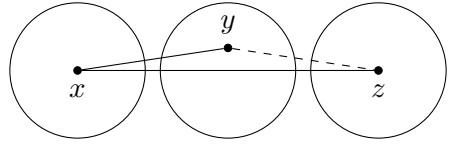
\begin{figure}[h]
        \[
            \begin{tikzpicture}
                \tikzstyle{vertex}=[circle, draw, fill=black, inner sep=1pt]
                \draw (0, 0) node[vertex, label=below:{$x$}](x) {};
                \draw (2, 0.3) node[vertex, label=above:{$y$}](y) {};
                \draw (4, 0) node[vertex, label=below:{$z$}](z) {};
                \draw (0, 0) circle [radius=0.9];
                \draw (2, 0) circle [radius=0.9];
                \draw (4, 0) circle [radius=0.9];
                \draw[black, solid] (y) -- (x) -- (z);
                \draw[black, dashed] (y) -- (z);
            \end{tikzpicture}
        \]
        \caption{$\Xi_2$-subgraph}\label{fig:Xi2}
        \end{figure}
    \end{enumerate}
\end{proof}

Because two distinct $\Xi_2$-sungraphs have no less than 3 edges, than we can state

\begin{corollary}\label{Xi2Cor}
    If there are two distinct $\Xi_2$-subgraph in $\gen{E}$,
    then the number of garland does not exceed $2^{|E| - 1} + 2^{|E| - 3} - 1$.
\end{corollary}

\begin{lemma}\label{K211Lemma}
    If there is a garland of type $K(2, 1, 1)$ in $\gen{E}$,
    then the number of garlands does not exceed $13\cdot 2^{|E| - 5} - 1$.
\end{lemma}
\begin{proof}
    A continuable outside of edge set of such garland set of edges should be one of the following:
    empty, 5 one-edge sets,
    2 triangles,
    two garlands of type $K(2, 1)$, 2 pairs of nonadjancent edges and the edge aggregate of this
    garland.
\end{proof}

In addition to proof of lemma~\ref{K211Lemma} notice, that unconinuable outside of the edge aggregate of the garland $K(2, 1, 1)$ should be one of
the following:
\begin{itemize}
    \item 6 two-edge subsets, elements of which  is edges of the same triangle;
    \item 8 three-edge subsets, which does not contain triangles;
    \item 5 four-edge subsets.
\end{itemize}

All possible garlands, which edge aggregates contain no more than four edges, are shown in fig.~\ref{fig:dc}.
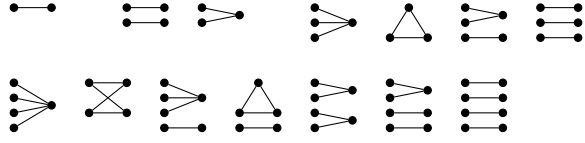
\begin{figure}[h]
\[
\begin{tikzpicture}
    \tikzstyle{vertex}=[circle, draw, fill=black, inner sep=1pt]
    \draw (0, 0) node[vertex] {};
    \draw (0.5, 0) node[vertex] {};
    \draw (0, 0) -- (0.5, 0);

    \draw (1.5, 0) node[vertex] {};
    \draw (2, 0) node[vertex]{};
    \draw (1.5, -0.2) node[vertex] {};
    \draw (2, -0.2) node[vertex] {};
    \draw (1.5, 0) -- (2, 0);
    \draw (1.5, -0.2) -- (2, -0.2);

    \draw (2.5, 0) node[vertex] {};
    \draw (2.5, -0.2) node[vertex] {};
    \draw (3, -0.1) node[vertex] {};
    \draw (2.5, 0) -- (3, -0.1) -- (2.5, -0.2);

    \draw (4, 0) node[vertex] (k31_1) {};
    \draw (4, -0.2) node[vertex] (k31_2) {};
    \draw (4, -0.4) node[vertex] (k31_3) {};
    \draw (4.5, -0.2) node[vertex] (k31_0) {};
    \foreach \x in {k31_1, k31_2, k31_3} {
        \draw (\x) -- (k31_0);
    }

    \draw (5.25, 0) node[vertex](t1) {};
    \draw (5, -0.4) node[vertex](t2) {};
    \draw (5.5, -0.4) node[vertex](t3) {};
    \draw (t1) -- (t2) -- (t3) -- (t1);

    \draw (6, 0) node[vertex] (k2_2_1) {};
    \draw (6, -0.2) node[vertex] (k2_2_2) {};
    \draw (6.5, -0.1) node[vertex] (k2_2_0) {};
    \draw (6, -0.4) node[vertex] {};
    \draw (6.5, -0.4) node[vertex] {};
    \draw (k2_2_1) -- (k2_2_0) -- (k2_2_2);
    \draw (6, -0.4) -- (6.5, -0.4);

    \draw (7, 0) node[vertex] {};
    \draw (7, -0.2) node[vertex] {};
    \draw (7, -0.4) node[vertex] {};
    \draw (7.5, 0) node[vertex] {};
    \draw (7.5, -0.2) node[vertex] {};
    \draw (7.5, -0.4) node[vertex] {};
    \draw (7, 0) -- (7.5, 0);
    \draw (7, -0.2) -- (7.5, -0.2);
    \draw (7, -0.4) -- (7.5, -0.4);

    \def\secondRowStartY{-1};
    \def\startX{0};
    \draw (\startX, \secondRowStartY) node[vertex](k4_1) {};
    \draw (\startX, \secondRowStartY-0.2) node[vertex](k4_2) {};
    \draw (\startX, \secondRowStartY-0.4) node[vertex](k4_3) {};
    \draw (\startX, \secondRowStartY-0.6) node[vertex](k4_4) {};
    \draw (\startX + 0.5, \secondRowStartY-0.3) node[vertex](k4_0) {};
    \foreach \x in {k4_1, k4_2, k4_3, k4_4} {
        \draw (\x) -- (k4_0);
    }

    \def\startX{1}
    \draw (\startX, \secondRowStartY) node[vertex](k22_1) {};
    \draw (\startX, \secondRowStartY-0.4) node[vertex](k22_3) {};
    \draw (\startX +0.5, \secondRowStartY) node[vertex](k22_2) {};
    \draw (\startX +0.5, \secondRowStartY-0.4) node[vertex](k22_4) {};
    \foreach \x in {k22_1, k22_3} {
        \foreach \y in {k22_2, k22_4} {
            \draw (\x) -- (\y);
        }
    }

    \def\startX{2};
    \draw (\startX, \secondRowStartY) node[vertex](k31_2_1) {};
    \draw (\startX, \secondRowStartY - 0.2) node[vertex](k31_2_2) {};
    \draw (\startX, \secondRowStartY - 0.4) node[vertex](k31_2_3) {};
    \draw (\startX + 0.5, \secondRowStartY - 0.2) node[vertex](k31_2_0) {};
    \draw (\startX, \secondRowStartY -0.6) node[vertex] {};
    \draw (\startX + 0.5, \secondRowStartY -0.6) node[vertex] {};
    \foreach \x in {k31_2_1, k31_2_2, k31_2_3} {
        \draw (\x) -- (k31_2_0);
    }
    \draw (\startX, \secondRowStartY - 0.6) -- (\startX + 0.5, \secondRowStartY -0.6);

    \def\startX{3}
    \draw (\startX + 0.25, \secondRowStartY) node[vertex](t2_1) {};
    \draw (\startX, \secondRowStartY - 0.4) node[vertex](t2_2) {};
    \draw (\startX + 0.5, \secondRowStartY -0.4) node[vertex](t2_3) {};
    \draw (\startX, \secondRowStartY -0.6) node[vertex] (tmp1) {};
    \draw (\startX + 0.5, \secondRowStartY - 0.6) node[vertex] (tmp2) {};
    \draw (t2_1) -- (t2_2) -- (t2_3) -- (t2_1);
    \draw (tmp1) -- (tmp2);

    \def\startX{4}
    \draw (\startX, \secondRowStartY) node[vertex] (k12_1_1) {};
    \draw (\startX, \secondRowStartY - 0.2) node[vertex] (k12_1_2) {};
    \draw (\startX + 0.5, \secondRowStartY -0.1) node[vertex] (k12_1_0) {};
    \draw (\startX, \secondRowStartY - 0.4) node[vertex] (k12_2_1) {};
    \draw (\startX, \secondRowStartY - 0.6) node[vertex] (k12_2_2) {};
    \draw (\startX + 0.5, \secondRowStartY -0.5) node[vertex] (k12_2_0) {};
    \foreach \i in {1, 2} {
        \foreach \j in {1, 2} {
            \draw (k12_\i_\j) -- (k12_\i_0);
        }
    }

    \def\startX{5}
    \draw (\startX, \secondRowStartY) node[vertex] (k12_1_1) {};
    \draw (\startX, \secondRowStartY - 0.2) node[vertex] (k12_1_2) {};
    \draw (\startX + 0.5, \secondRowStartY -0.1) node[vertex] (k12_1_0) {};
    \draw (\startX, \secondRowStartY -0.4) node[vertex] (tmp1_1) {};
    \draw (\startX + 0.5, \secondRowStartY -0.4) node[vertex] (tmp1_2) {};
    \draw (\startX, \secondRowStartY -0.6) node[vertex] (tmp2_1) {};
    \draw (\startX + 0.5, \secondRowStartY -0.6) node[vertex] (tmp2_2) {};
    \foreach \i in {1, 2} {
        \draw (tmp\i_1) -- (tmp\i_2);
        \draw (k12_1_\i) -- (k12_1_0);
    }

    \def\startX{6}
    \foreach \i in {0, 0.2, 0.4, 0.6} {
        \draw (\startX, \secondRowStartY - \i) node[vertex](local1) {};
        \draw (\startX + 0.5, \secondRowStartY - \i) node[vertex](local2) {};
        \draw (local1) -- (local2);
    }

\end{tikzpicture}
\]

    \caption{All possible garlands, which edge aggregate contains no more than 4 edges}\label{fig:dc}
\end{figure}

Part of the graph $K(v)$ is called \textit{active},
if there is a vertex in this part, which is incedent to some edge from $E$.

\begin{lemma}\label{SixEdgesLemma}
    Let every active part of $K(v)$ contains at least 4 vertices and $|E| = 6$. Let also each garland of cardinality $p$ destroys no more than $p-1$ parts
    of the graph $K(v)$. Then either $\gen{E}$ is an intresting garland of type $K(6, 1)$ and contains exactly 63 intresting garlands; or subgraph $\gen{E}$
    contains no more than 33 intresting garlands.
\end{lemma}

\begin{proof}
    Notice, that any garland of cardinality one is interesting, because it can not destroy any part.

    If there is a garland of type $K(6, 1)$ in the graph $\gen{E}$,
    then $\gen{E}$ is an interesting garland of type $K(6, 1)$ and
    it contains exactly 63 interesting garlands.

    Assume there is no garland of type $K(6, 1)$ in $\gen{E}$.
    Assume, that there is garlalnd of type $K(5, 1)$ in the graph $\gen{E}$.
    Denote an edge, which does not lay in $K(5, 1)$, as $e$.
    Then either $e$ is incedent to the vertex of degree one of garland $K(5, 1)$,
    incedent to the vertex of degree five of garland $K(5, 1)$ or non incident to any vertex of garland $K(5, 1)$.
    In all this cases edge $e$ lays in no more than two interesting garlands
    (one-edge garlnand, and, maybe six-edge garland of type $K(5, 1)\disjoint K(1, 1)$ or two-edge
    garland of type $K(2, 1)$.
    Then there are no more than $2^5 - 1 + 2 = 33$ garlands.
    \begin{figure}[!h]
        \includegraphics[width=0.9\textwidth]{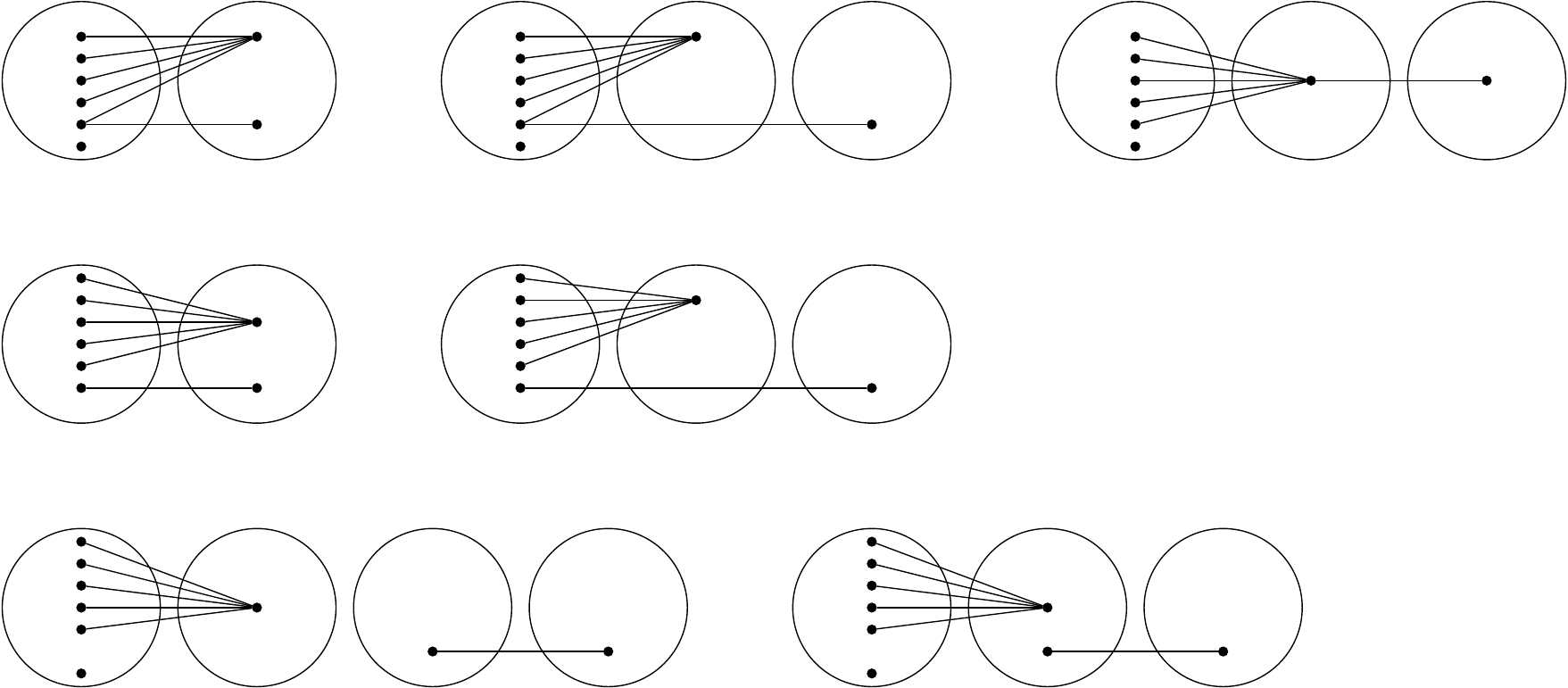}
        \caption{A mutual position of the garland $K(5, 1)$ and the edge $e$}\label{fig:K51}
    \end{figure}

    An interesting four-edge garland can be one of the following types:
    $K(4, 1)$, $K(2, 1)\disjoint K(2, 1)$, $K(3, 1)\disjoint K(1, 1)$ and $K(2, 2)$,
    such as garlands of type \\
    $K(1, 1, 1)\disjoint K(1,1)$, $K(2, 1)\disjoint K(1, 1)\disjoint K(1, 1)$ and $K(1, 1)\disjoint K(1, 1)\disjoint K(1, 1) \disjoint K(1, 1)$
    can not destroy necessary number of parts to be an interesting garland.

    Let there is no garlands of type $K(5, 1)$ and $K(6, 1)$ in $\gen{E}$.
    Assume that there is an interesting garland of type $K(4, 1)$,
    there can not be more than one of them. Denote the part,
    which contains all vertices of degree 1 of the garland $K(4, 1)$, as $V_1$.
    Notice, that if there is a garland of type $K(3, 1)$ which does not
    lay into any garland of type $K(4, 1)$,
    then it located as shown in fig.~\ref{fig:K41}
    (because this two garlands can not have common edges).
    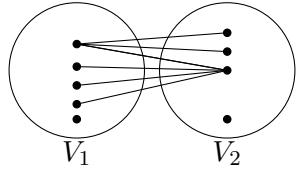
\begin{figure}[h]
\[
\begin{tikzpicture}
    \tikzstyle{vertex}=[circle, draw, fill=black, inner sep=1pt]
    \draw (0, -0.45) node[vertex](x11){} {};
    \draw (0, -0.2) node[vertex](x12) {};
    \draw (0, 0.05) node[vertex](x13) {};
    \draw (0, 0.35) node[vertex](x14) {};

    \draw (2, 0) node[vertex](y11) {};
    \draw (2, 0.25) node[vertex](y12) {};
    \draw (2, 0.5) node[vertex](y13) {};
    \draw (0, 0) circle[radius=0.9] {};
    \draw node at (0, -1.1) {$V_1$};
    \draw (2, 0) circle [radius=0.9];
    \draw node at (2, -1.1) {$V_2$};
    \foreach \x in {x11, x12, x13, x14} {
        \draw[black, solid] (\x) -- (y11);
    }

    \foreach \y in {y11, y12, y13} {
        \draw[black, solid] (\y) -- (x14);
    }

    \draw (0, -0.65) node[vertex] {};
    \draw (2, -0.65) node[vertex] {};

\end{tikzpicture}
\]
        \caption{Garlands $K(4, 1)$ and $K(3, 1)$}\label{fig:K41}
    \end{figure}

    In this case, a number of interesting garlands is equal to $2^4 - 1 + 2^3 - 1 = 15 + 7 = 22$.

    Consider the case, when every garland of type $K(3, 1)$ lay into some garland of type $K(4, 1)$.

    Assume, that there is an interesting garland $G'$ of type $K(3, 1)\disjoint K(1, 1)$.
    Notice, that the garland $G'$ should destroy some part.
    It can destroy part which contains only 1, 2, 3, or 4 vertices.
    By lemma statement, each part contains at last four vertices,
    so it should destroy four-vertex part and this part should contain all three vertices of degree
    one of the graph $K(3, 1)$ of the garland $G$,
    but this vertices lay into a part which contains more than 4 vertices, that is a contradiction.
    Consequently, there is no interesting garland of type $K(3, 1)\disjoint K(1, 1)$.

    Notice, that any interesting garland of type $K(2, 1)\disjoint K(2, 1)$ should destroy some part.
    It can destroy a part which contains at most four vertices.
    By lemma statement, any active part contains at least four vertices.
    If a part is destroyed by a garland of type $K(2, 1)\disjoint K(2, 1)$,
    then all four vertices of degree one should lay in this part, but this vertices
    lay in the part $V_1$, which contains at least 5 vertices, and this is a contradiction.

    Also notice, that there are no more than one interesting garland of type $K(2, 2)$,
    because they should have two edges from garland of type $K(4, 1)$.
    An interesting three-edge garland should be triangle or have type $K(3, 1)$.
    There are exactly four garlands of type $K(3, 1)$.
    There is no more than one triangle, because there are only two edges outside of $K(4, 1)$.
    Estimate a number of interesting two-edge garlands.
    The number of them, which lay inside garland of type $K(4, 1)$, is exactly $\binom{4}{2} = 6$.
    There are no more than three two-edge garlands, which contain edges not from $K(4, 1)$:
    there is no more than one garland, whose edges lay outside of $K(4, 1)$
    and there are no more than two garlands, which contains exactly one edge which does not lay in $K(4, 1)$.
    So, there are no more than $6 + 9 + (1 + 4) + 2 + 5 + 1 = 28$ interesting garlands.

    Assume, that there is no garlands of type $K(N, 1)$ for all $N\geq 4$.
    Then there is no more than three garlands of types $K(2, 2)$ and $K(2, 1)\disjoint K(2, 1)$.
    There is no more than two garlands of type $K(3, 1)$, because they can not have more than one common edge.
    Let $k$ be the greatest number of entries of garland of type $K(3, 1)$ in garlands of type $K(3, 1)\disjoint K(1,1)$.
    Then there are at least $3k$ pairs nonadjacent edges.
    Therefore, by lemma~\ref{EdgePairsLemma} one can deduce that $\mu_2 + \xi_3 \leq \binom{6}{2} - 3k = 15 - 3k$.
    Consequently,
    there are exactly 6 one-edge interesting garlands,
    there are exactly $\mu_2$ two-edge interesting garlands,
    there are no more than $\xi_3 + 3$ three-edge interesting garlands,
    there are no more than $2k + 3$ four-edge interesting garlands,
    there are no more than $\binom{6}{5} = 6$ five-edge interesting garlands,
    there are no more than $\binom{6}{6} = 1$ six-edge interesting garlands,
    so, there are no more than $6 + \mu_2 + \xi_3 + 3 + 2k + 2 + 6 + 1 \leq 18 + 15 - 3k + 2k \leq 33$ interesting garlands.
\end{proof}

The following lemma was proved in \cite{Gein}.

\begin{lemma}\label{OneStepLemma}
    Let $u=(u_1, \ldots,u_i,\ldots,u_j,\ldots u_t)\rightarrow v=(\ldots, u_i - 1, \ldots u_j + 1,\ldots)$
    be an elementary transformation of partition $u$ and element $u_t \geq 2$.
    Then graphs $K(u)$ and $H$ are not chromatically equivalent.
\end{lemma}

\section{Case $r=0$}
The lowest levels of the lattice $NPL(n, 3)$ in the case when $n$ is divided by 3 is shown on fig.~\ref{fig:NPL0}.
By analogy with \cite{KorolevaR0},
the difference of the number of edges is placed over cover relation, and the difference of the invariant $pt$
is placed under cover relation.
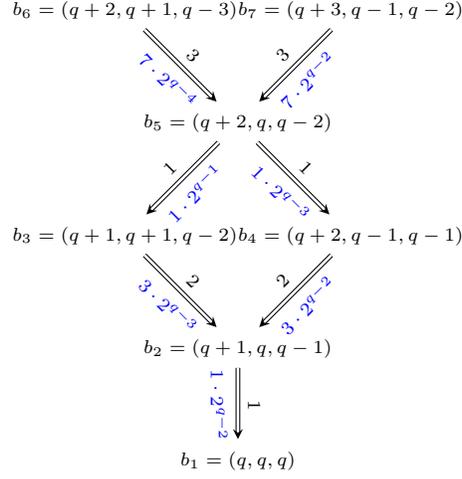
\begin{figure}
\[
\begin{tikzpicture}
    \tikzstyle{every node}=[font=\tiny]
    \tikzstyle{arr}=[thin,double,double distance=1pt,->, >=stealth,shorten >=1pt]
    \tikzstyle{edgediff}=[sloped,draw=none,midway,above=0.000000pt,text=black]
    \tikzstyle{ptdiff}=[sloped,draw=none,midway,below=0.000000pt,text=blue]
    \draw (9.00, 0.00) node(b_0_0) {$b_{1}=(q,q,q)$};
    \draw (9.00, 1.50) node(b_1_0) {$b_{2}=(q+1,q,q-1)$};
    \draw (7.50, 3.00) node(b_2_0) {$b_{3}=(q+1,q+1,q-2)$};
    \draw (10.50, 3.00) node(b_2_1) {$b_{4}=(q+2,q-1,q-1)$};
    \draw (9.00, 4.50) node(b_3_0) {$b_{5}=(q+2,q,q-2)$};
    \draw (7.50, 6.00) node(b_4_0) {$b_{6}=(q+2,q+1,q-3)$};
    \draw (10.50, 6.00) node(b_4_1) {$b_{7}=(q+3,q-1,q-2)$};
    \draw[arr](b_4_0)--(b_3_0)node[edgediff]{$3$}node[ptdiff]{$7\cdot2^{q-4}$};
    \draw[arr](b_4_1)--(b_3_0)node[edgediff]{$3$}node[ptdiff]{$7\cdot2^{q-2}$};
    \draw[arr](b_3_0)--(b_2_0)node[edgediff]{$1$}node[ptdiff]{$1\cdot2^{q-1}$};
    \draw[arr](b_3_0)--(b_2_1)node[edgediff]{$1$}node[ptdiff]{$1\cdot2^{q-3}$};
    \draw[arr](b_2_0)--(b_1_0)node[edgediff]{$2$}node[ptdiff]{$3\cdot2^{q-3}$};
    \draw[arr](b_2_1)--(b_1_0)node[edgediff]{$2$}node[ptdiff]{$3\cdot2^{q-2}$};
    \draw[arr](b_1_0)--(b_0_0)node[edgediff]{$1$}node[ptdiff]{$1\cdot2^{q-2}$};
\end{tikzpicture}
\]
    \caption{The lowest levels of the lattice $NPL(n, 3)$ in case when $n$ is divided by 3}\label{fig:NPL0}
\end{figure}
There are only two elements of height 4: $(q + 2, q+1, q-3)$ and $(q + 3, q-1, q-2)$.
Chromatic uniqueness of a graph $K(q+2, q+1, q-3)$ if $q\geq 5$ follows from the main result of \cite{Gein}.

\begin{proposition}
    A graph $K(q+3,q-1,q-2)$ is chromatically unique, if $q\geq 4$.
\end{proposition}
\begin{proof}
    Let graphs $K(q+3, q-1, q-2) = K(u)$ and $H$ are chromatically equivalent
    and the graph $H$ is obtained from graph $K(v)$ by deleting the edge set $E$.
    Consider the cases for the partition $v$.

    Cases when $v=(q+2, q, q-2)$ and $v=(q+2, q-1, q-1)$ contradict with lemma~\ref{OneStepLemma}.

    \textbf{Case 1.} Let $v=(q+1,q+1,q-2)$. Then  $|E|=4$ and, by lemma~\ref{lemma:PT_Main}, one can obtain
            $$ \Delta pt(H, K(v)) = 7\cdot 2^{q-2} + 2^{q-1} = 9 \cdot 2^{q-2} \leq 15, $$
            which implies that $q\leq 2$, which is a contradiction.

    \textbf{Case 2.} Let $v=(q+1,q, q-1)$. Then  $|E|=6$ and using lemma \ref{lemma:PT_Main} one can deduce,
    that
            $$
            \Delta pt(H, K(v)) = 9\cdot 2^{q-2} + 3\cdot 2^{q-3} = 36\cdot 2^{q-4} + 6\cdot 2^{q-4} = 42\cdot 2^{q-4} \leq 63,
            $$
            which implies that $q=4$, $\Delta pt(H, K(v)) = 42$ and $v = (5, 4, 3)$.
            Compute the difference of the invariant $I_3$:
            \begin{gather*}
                \Delta I_3(K(v),K(u)) = 3(q-2) + q + 2 + 2(q-1) = 6q - 6  = 18, \\
                \Delta I_3(K(v), H) = \xi_1 - \xi_2 - 2\xi_3 = 18, \\
                \xi_1 = 3e_{12} + 4e_{13} + 5e_{23} = 3|E| + e_{13} + 2e_{23} = 18 + e_{13} + 2e_{23}, \\
                e_{13} + 2e_{23} = \xi_2 + 2\xi_3
            \end{gather*}

            Let $e_{12} = 6$. Then $\gen{E}$ is a subgraph of the complete bipartite graph of type $K(5, 4)$.
            It should contain exactly 42 interesting garlands, and this contradicts with
            lemma~\ref{SixEdgesLemma}.
            Consequently, $\xi_2 + 2\xi_3 > 0$.

            Let $\xi_3 > 0$. Then by lemma~\ref{TriagLemma} the number of interesting garland does not exceed $5\cdot 8 - 1 = 39 < 42$,
            which is a contradiction.
            Consequently, $\xi_3 = 0$.

            Let $\xi_2\geq 2$. Then by corollary~\ref{Xi2Cor} the number of interesting garlands does not exceed
            $63 - 32 + 2^{6 - 3} = 39 < 42$, which is impossible.

            Therefore, $\xi_2 = 1$, and one can obtain that  $e_{13} + 2e_{23} = 1$,
            consequently, $e_{23} = 0, e_{13} = 1, e_{12} = 5$.
            Denote a single edge between parts $V_1$ and $V_3$ as $e$.
            Because of $\xi_2 = 1$, there is exactly one edge in  $\gen{E}$, which is adjacent with $e$.
            Denote this edge as $f$.
            Then there are no more $2^{|E \setminus\{e\}|} - 1 = 2^5 - 1 = 31$ garlands,
            whose edges lay in $E\setminus \{e\}$.
            Notice, that any non-one-edge garland, which contains the edge $e$, can not have cardinality one,
            so should destroy part $V_1$ and
            notice, that this garland can not contain the edge $f$, therefore,
            there are no more than one such garlands. Consequently,
            there are no more than $31 + 1 + 1 = 33 < 42$ garlands, which is a contradiction.

    \textbf{Case 3.} Let $v=(q, q, q)$. Then $|E| = 7$ and by lemma~\ref{lemma:PT_Main} one can deduce that
            $$
            \Delta pt(H, K(v)) = 42\cdot 2^{q-4} + 2^{q-2} = 46\cdot 2^{q-4} \leq 127,
            $$
            so $q=4$ or $q=5$. Compute the difference of the invariant $I_3$:
            \begin{gather*}
                \Delta I_3(v,u) = 6q-6 + q = 7q - 6 \\
                \Delta I_3(v,H) = \xi_1 - \xi_2 - 2\xi_3 = 7q-6\\
                \xi_1 = qe_{12} + qe_{13} + qe_{23} = q|E| = 7q \\
                6 = \xi_2 + 2\xi_3,
            \end{gather*}
            therefore,  $\xi_3 > 0$ or $\xi_2 = 6$.

            Assume, that $q=5$. In this case $\Delta pt(H, K(v)) = 92$. If $\xi_3 > 0$, then by lemma~\ref{TriagLemma}
            a number of interesting garlands does not exceed $5\cdot 16 - 1 = 79$.
            If $\xi_2 = 6$, then by corollary~\ref{Xi2Cor} the number of interesting garlands does not exceed
            $2^7 - 1 - 2\cdot 2^5 + 2^{7 - 3} = 127 - 64 + 16 = 79$, which is contradiction.

            Now consider the case when $q=4$. In this case $v=(4, 4, 4)$ and $\Delta pt(H, K(v)) = 46$.

            Assume that $\xi_2 = 0$. Then $\xi_3 = 3$, which implies that there are two distinct triangles in $\gen{E}$, which have a common edge.
            Denote set of all edges of this triangles as $E'$, and notice, that $|E'| = 5$.
            Since the graph $K(v)$ is tripartite, set $E'$ is an edge aggregate of the garland of type $K(2, 1, 1)$.
            Consider an arbitrary continuable outside of $'E'$ subset $E_1\subset E'$ beside a couple of nonadjacent edges
            (see proof of lemma~\ref{K211Lemma}), there are 11 such subsets.
            The number of garlands, such that an intersection of their edge aggregates with set of edges of garland $K(2, 1, 1)$ is equal to $E_1$, is not
            greater than $2^2 = 4$.
            It is left to estimate the number of such garlands $G'$,
            that an intersection of their edge aggregate with set of edges of garland $K(2, 1, 1)$
            is equal to a pair of nonadjacent edges.
            Notice, that in this case the cardinality of the garland $G'$ is not less than 2,
            since if it is equal to 1, then it should contain
            another edge from $K(2, 1, 1)$.
            Consequently, the garland $G'$ should destroy some part, therefore,
            it should contain as least 4 edges and
            two of them does not lay in $K(2, 1, 1)$, so there are no more than 2 such garlands,
            because there are only two pairs of nonadjacent edges in $K(2, 1, 1)$.
            So, there are no more than $11\cdot 4 - 1 + 2 = 45 < 46$, which is a contradiction.
            Consequently, $\xi_2 > 0$. Since $\xi_2 = 6 - 2\xi_3$ is an even number,
            one can obtain that $\xi_2 \geq 2$.

            Notice, that there are no garlands in $\gen{E}$, whose edge aggregates contains exactly 7 edges,
            because such garlands should contain an uncontinuable subset ---
            a set of edges of some $\Xi_2$-subgraph.
            Also notice, that there are no more than 1 six-edge garlands
            (because edges $f$ and $e$ form a $\Xi_2$-subgraph,
            so a six-edge garland should contain exactly one of them,
            because all edges can not simultaneously lay in the same garland.
            If there are two nonintersecting $\Xi_2$-subgraphs,
            then there are no six-edge garlands; if $f$ -- is a common edge of two distinct
            $\Xi_2$-subgraph, then it can not lay in any six-edge garland).
            There are no more than $\binom{7}{5} - \binom{5}{3} = 21 - 10 = 11$ five-edge
            garland, since there are $\binom{5}{3}$ five-element subsets, which contain edges of a certain $\Xi_2$-subgraph.

            An interesting three-edge garland should be triangle or has type $K(3, 1)$.

            Estimate the number of garlands of type $K(3, 1)$.
            Since there are no garland of type $K(4, 1)$, because such garland destroys a part,
            which is impossible, any two garlands of type $K(3, 1)$ have no more than one common edge; therefore, there are no more than three such
            garlands.

            An interesting four-edge garland should have one of the three types:
            $K(3,1)\disjoint K(1,1),\\ K(2,1)\disjoint K(2,1),\, K(2, 2)$.
            There are no more than two garlands of type  $K(2, 2)$
            and there are no more than three garlands of type ~$K(2, 1)\disjoint K(2, 1)$.

            Each garland of type $K(3, 1)\disjoint K(1, 1)$ contain inside itself a garland of type $K(3, 1)$.
            Let $k$ be the greatest number of entries of garlands $K(3, 1)$ in interesting garlands of type $K(3,1)\disjoint K(1,1)$.
            Then there are at least $3k$ pairs of nonadjacent edges in $\gen{E}$, so by lemma~\ref{EdgePairsLemma}, one can deduce that
            $\xi_2 + \mu_2 + 3\xi_3 \leq \binom{7}{2} - 3k = 21 - 3k$,
            therefore, $\mu_2 + \xi_3 \leq 15 - 3k$.
            Then there are no more than $3k$ garlands of type $K(3, 1)\disjoint K(1, 1)$.
            Then there are no more than $7 + \mu_2 + (3 + \xi_3) + (3k + 5) + 11 + 1 \leq 27 + 15 - 3k + 3k  = 42 < 46$ interesting garlands,
            which is a contradiction.
\end{proof}

\section{Case $r=1$}

The lowest levels of the lattice $NPL(n, 3)$, when $n$ is equal 1 modulo 3, are shown on fig.~\ref{fig:NPL1}.
As in the previous case, the difference of the number of edges is placed over cover relation, and
the difference of the invariant $pt$ is placed under cover relation.
\begin{figure}
\[
\begin{tikzpicture}
     \tikzstyle{every node}=[font=\tiny]
     \tikzstyle{arr}=[thin,double,double distance=1pt,->, >=stealth,shorten >=1pt]
     \tikzstyle{edgediff}=[sloped,draw=none,midway,above=0.000000pt,text=black]
     \tikzstyle{ptdiff}=[sloped,draw=none,midway,below=0.000000pt,text=blue]
     \draw (9.00, 0.00) node(b_0_0) {$b_{1}=(q+1,q,q)$};
     \draw (7.50, 1.50) node(b_1_0) {$b_{2}=(q+1,q+1,q-1)$};
     \draw (9.00, 3.00) node(b_2_0) {$b_{3}=(q+2,q,q-1)$};
     \draw (7.50, 4.50) node(b_3_0) {$b_{4}=(q+2,q+1,q-2)$};
     \draw (10.50, 4.50) node(b_3_1) {$b_{5}=(q+3,q-1,q-1)$};
     \draw (6.00, 6.00) node(b_4_0) {$b_{6}=(q+2,q+2,q-3)$};
     \draw (9.00, 6.00) node(b_4_1) {$b_{7}=(q+3,q,q-2)$};
     \draw[arr](b_4_0)--(b_3_0)node[edgediff]{$4$}node[ptdiff]{$15\cdot2^{q-4}$};
     \draw[arr](b_4_1)--(b_3_0)node[edgediff]{$2$}node[ptdiff]{$3\cdot2^{q-1}$} ;
     \draw[arr](b_4_1)--(b_3_1)node[edgediff]{$1$}node[ptdiff]{$1\cdot2^{q-3}$} ;
     \draw[arr](b_3_0)--(b_2_0)node[edgediff]{$2$}node[ptdiff]{$3\cdot2^{q-3}$} ;
     \draw[arr](b_3_1)--(b_2_0)node[edgediff]{$3$}node[ptdiff]{$7\cdot2^{q-2}$} ;
     \draw[arr](b_2_0)--(b_1_0)node[edgediff]{$1$}node[ptdiff]{$1\cdot2^{q-1}$} ;
     \draw[arr](b_1_0)--(b_0_0)node[edgediff]{$1$}node[ptdiff]{$1\cdot2^{q-2}$} ;
\end{tikzpicture}
\]
\caption{The lowest levels of the lattice $NPL(n, 3)$ in case, when $n$ is equal 1 modulo 3}\label{fig:NPL1}
\end{figure}
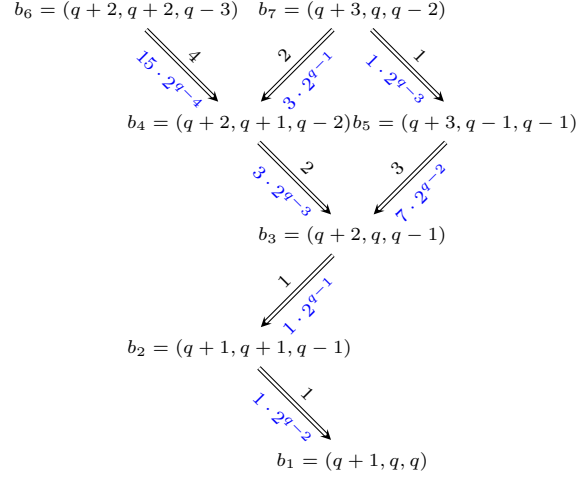
There are only two elements of height 4: $(q+2, q+2, q-3)$ and $(q+3, q, q-2)$.
Chromatic uniqueness of the graph $K(q + 2, q+2, q-3)$ if $q\geq 5$ follows from theorem 1 in \cite{Chia}.

\begin{proposition}
    A graph $K(q +3, q, q-2)$ is chromatically unique if $q\geq 4$.
\end{proposition}
\begin{proof}
    Assume graphs $K(q+3, q, q-2) = K(u)$ and $H$ are chromatically equivalent,
    and the graph $H$ is obtained from a graph $K(v)$ by deleting edge set $E$.
    Consider possible cases for partition $v$.

    Cases $v=(q+2, q+1, q-2), v=(q+3, q-1, q-1)$ and $v=(q+2, q, q -1)$ are contradict with lemma~\ref{OneStepLemma}.

    \textbf{Case 1.} Let $v=(q+1, q+1, q-1)$. Then $|E|=5$ and by lemma~\ref{lemma:PT_Main} one can obtain, that
            $$
            \Delta pt(H, K(v)) = 2^{q-3} + 7\cdot 2^{q-2} + 2^{q-1} = (2 + 28 + 8)\cdot 2^{q-4} = 38\cdot 2^{q-4} \leq 2^5 - 1,
            $$
            which is a contradiction since $q\geq 4$.

    \textbf{Case 2.} Let $v=(q+1, q, q)$, In this case $|E|=6$ and using lemma~\ref{lemma:PT_Main} one can deduce that
            $$
            \Delta pt(H, K(v)) = 38\cdot 2^{q-4} + 2^{q-2} = (38 + 4)2^{q-4} = 42\cdot 2^{q-4} \leq 2^6 - 1,
            $$
            which implies that $q=4$ and $\Delta pt(H, K(v)) = 42$, which is a contradiction with lemma~\ref{SixEdgesLemma}.
\end{proof}

\textbf{Acknowledgment.} The author is grateful to V.A. Baransky for attention and remarks, which are assist to significant improvement of this paper.

\end{document}